\documentclass[reqno]{amsart}
\usepackage{amssymb,amscd,verbatim, amsthm,graphicx, color}
\usepackage{fancybox}
\usepackage{longtable}

\newcommand \fk[1]{{{\mathfrak #1}}}
\newcommand \C[1]{{\mathcal #1}}

\newcommand \wti[1]{{\widetilde {#1}}}

\newcommand\fg{\mathfrak g}

\newcommand \bC{{\mathbb C}}

\newcommand \bR{{\mathbb R}}
\newcommand \bZ{{\mathbb Z}}

\newcommand\CC{{\C C}}

\newcommand\CP{{\C P}}

\newcommand\ep{{\epsilon}}

\newcommand\om{{\omega}}
\newcommand\al{{\alpha}}

\newcommand\fh{{\mathfrak h}}
\newcommand\ft{{\mathfrak t}}

\newtheorem{theorem}{Theorem}[subsection]

\newtheorem{corollary}[theorem]{Corollary}
\newtheorem{lemma}[theorem]{Lemma}
\newtheorem{proposition}[theorem]{Proposition}
\newtheorem{definition}[theorem]{Definition}
\newtheorem{remark}[theorem]{Remark}

\newcommand\Hom{\operatorname{Hom}}
\newcommand\Ind{\operatorname{Ind}}

\newcommand\tr{\operatorname{tr}}
\newcommand\pr{\operatorname{pr}}

\newcommand\triv{\mathsf{triv}}
\newcommand\sgn{\mathsf{sgn}}

\numberwithin{equation}{subsection}

\begin{document}

\title[Spin representations of Weyl groups]{Spin representations of Weyl groups and the Springer correspondence} 

\author{Dan Ciubotaru}
        \address[D. Ciubotaru]{Dept. of Mathematics\\ University of
          Utah\\ Salt Lake City, UT 84112}
        \email{ciubo@math.utah.edu}


\begin{abstract}
We give a common framework for the classification of projective spin irreducible representations of a Weyl group, modeled after the Springer correspondence for ordinary representations.  
\end{abstract}

\maketitle

\setcounter{tocdepth}{1}

\section{Introduction}\label{sec:0}

Let $\Phi=(V,R,V^\vee,R^\vee)$ be a semisimple crystallographic $\bR$-root system (see
\S\ref{sec:1}) with Weyl group $W$ and a choice of positive roots $R^+$. Assume that $V$ is endowed with a
$W$-invariant inner product $\langle~,~\rangle$, and define the dual inner product on $V^\vee$, denoted by $\langle~,~\rangle$ as well. 
 The Weyl group $W$ is a finite subgroup of
$\mathsf{O}(V)$ and therefore, one can consider the double cover $\wti
W$ of $W$ in $\mathsf{Pin}(V)$, a double cover of
$\mathsf{O}(V)$. A classical problem is to classify the irreducible
genuine $\wti W$-representations (i.e., the representations that do not
factor through $W$). This is known case
by case, and goes back to Schur (\cite{Sc}), in the case of $\wti S_n$, and was
completed about $30$ years ago for the other root systems  by Morris, Read, Stembridge and others (see
\cite{Mo,Mo2, Re,St} and the references therein).  In this paper, we attempt to
unify these classifications in a common
framework, based on Springer theory (\cite{Sp,L}) for ordinary
$W$-representations. Our point of view is motivated by the
construction of the Dirac operator for graded affine Hecke algebras
(\cite{BCT}).

The group $\wti W$ is generated by certain elements $\wti s_\al$ of order $4$,
 $\al\in R^+$, with relations similar to the Coxeter
presentation of $W$ (see \S\ref{Wspin}). Let $\check\al\in R^\vee$ denote the coroot corresponding to the root $\al\in R.$ The starting observation is the
existence of a remarkable central element
$\Omega_{\wti W}\in \bC[\wti W]$ (see \S\ref{casimir}):
\begin{equation}\label{omintro}
\Omega_{\wti W}=\sum_{\substack{\al>0,\beta>0\\s_\al(\beta)<0}} 
 |\check\al| |\check\beta| ~\wti s_\al \wti s_\beta;
\end{equation}
this element, rather surprisingly, behaves like an analogue of the Casimir element for
a Lie algebra. Every irreducible $\wti W$-representation $\wti\sigma$
acts by a scalar $\wti\sigma(\Omega_{\wti W})$ on $\Omega_{\wti W}.$
For example, $\wti W$ has one (when $\dim V$ is even) and two (when
$\dim V$ is odd) distinguished irreducible representations, which we
call spin modules (\S\ref{s:spinmod}). If $S$ is one such spin module,
then $S(\Omega_{\wti W})=\langle2\check\rho,2\check\rho\rangle$, where $\check\rho=\frac
12\sum_{\al\in R^+}\check\al.$

Before stating the main result, we need to introduce more notation. Let $\fg$ be the complex semisimple Lie algebra with root system $\Phi$ and Cartan subalgebra $\fh=V^\vee\otimes_\bR \bC$, and let
$G$ be the simply connected Lie group with Lie algebra $\fg$. Extend the inner product from $V^\vee$ to $\fh.$ Let us denote by $\C T(G)$ the set of
$G$-conjugacy classes of Jacobson-Morozov triples $(e,h,f)$ in $\fg$. We set:
\begin{equation}
\C T_0(G)=\{[(e,h,f)]\in \C T(G): \text{ the centralizer of }\{e,h,f\} \text{
  in }\fg\text{ is a toral subalgebra}\}.
\end{equation}
For example, if $G=\mathsf{SL}(n)$, the nilpotent elements $e$ that occur in $\C T_0(G)$ are those whose Jordan canonical form has all parts distinct. In general, every distinguished nilpotent element $e$ (in the sense of Bala-Carter \cite{Ca2}) appears in $\C T_0(G)$.

For every class in $\C T(G)$, we may (and will) choose a representative $(e,h,f)$ such that $h\in\fh.$
For every
nilpotent element $e$, let $A(e)$ denote the A-group in $G$, and let
$\widehat {A(e)}_0$ denote the set of representations of $A(e)$ of
Springer type. For every $\phi\in \widehat{A(e)}_0$, let
$\sigma_{(e,\phi)}$ be the associated Springer representation (see \S\ref{ellW}).
Normalize the Springer correspondence so that
$\sigma_{0,\text{triv}}=\sgn$.

There is an equivalence relation $\sim$ on the space $\widehat{\wti
  W}_{\mathsf{gen}}$ of genuine irreducible $\wti W$-representations:
$\wti\sigma\sim\wti\sigma\otimes\sgn$; here, $\sgn$ is the sign $W$-representation.

\begin{theorem}\label{t:intro} There is a surjective map
\begin{equation}
\Psi:\widehat{\wti W}_{\mathsf{gen}} \longrightarrow \C T_0(G),
\end{equation}
with the following properties:
\begin{enumerate}
\item If $\Psi(\wti\sigma)=[(e,h,f)]$, then we have
\begin{equation}
\wti\sigma(\Omega_{\wti W})=\langle h,h\rangle,
\end{equation}
where $\Omega_{\wti W}$ is as in (\ref{omintro}).
\item   Let $(e,h,f)\in \C T_0(G)$ be given. For every  Springer representation $\sigma_{(e,\phi)}$,
  $\phi\in\widehat {A(e)}_0$, and every spin $\wti W$-module $S$,
  there exists $\wti \sigma\in \Psi^{-1}[(e,h,f)]$ such that $\wti\sigma$ appears with
  nonzero multiplicity in the tensor product
  $\sigma_{(e,\phi)}\otimes S$. Conversely, for every $\wti\sigma\in \Psi^{-1}[(e,h,f)]$, there exists a
  spin $\wti W$-module $S$ and a Springer representation
  $\sigma_{(e,\phi)}$, such that $\wti\sigma$ is contained in
  $\sigma_{(e,\phi)}\otimes S.$
\item If $e$ is distinguished, then properties (1) and (2) above
  determine a  bijection:
\begin{equation}\label{dist}
\Psi^{-1}([e,h,f])/_\sim \longleftrightarrow \{\sigma_{e,\phi}: \phi\in \widehat {A(e)}_0\}.
\end{equation}
\end{enumerate}
\end{theorem}
Since $\triv(\Omega_{\wti W})=\sgn(\Omega_{\wti W})$, 
Theorem \ref{t:intro}(1) says in particular that any two associate
genuine $\wti W$-types $\wti\sigma_1\sim\wti\sigma_2$ lie in the same fiber of $\Psi$. This is why we need to quotient by $\sim$ in (\ref{dist}). It is natural to ask if one could reformulate (2) in the theorem so
that a bijection like (\ref{dist}) (with certain appropriate quotients in the
right hand side) holds for  non-distinguished
$e\in \C N_0(\fg)$. This is almost always the case, but there are
counterexamples, e.g., Remark \ref{counter}.

We should make clear that
while the main
result may appear close to a Springer type classification for $\wti W$, we do not
provide here a geometric construction for genuine
representations of $\wti W$. As we explain in \S\ref{s:3}, this
classification fits in the setting of elliptic
representation theory of $W$ (Reeder \cite{R}) and $\wti W$, and
its connection with nilpotent orbits. A different relation, between elliptic conjugacy classes in $W$ and a family of nilpotent orbits (``basic'') is presented in a recent paper by Lusztig \cite{L2}.

There are two directions in which one can generalize Theorem
\ref{t:intro}. Firstly,  it is apparent that one can extend these
results to the generalized Springer
correspondence (\cite{L}), by using a Casimir element $\Omega_{\wti W,c}$ for
an appropriate parameter function $c:R^+\to \bZ.$ We present the details in \S\ref{s:gen}, the main result being Theorem \ref{t:gen}, which is the exact analogue of Theorem \ref{t:intro}.

Secondly, an analogous
correspondence should hold, even in the absence of nilpotent orbits, for
non-crystallographic root systems, and more generally, for complex
reflection groups. There, one should be able to substitute the nilpotent
orbits and Springer representations in the right hand side of the
correspondence in Theorem \ref{t:intro}, with the space of elliptic tempered
modules (in the sense of \cite{R,OS})
for the corresponding graded Hecke algebra and their ``lowest $W$-types''. This problem will be considered elsewhere.  


\smallskip

\noindent{\bf Acknowledgements}. I am grateful to the referee for useful
comments and particularly for pointing out a mistake in a previous
version of section \ref{ellWtil}. This research was supported by NSF-DMS 0968065 and NSA-AMS 081022.

\section{Preliminaries}\label{sec:1}

\subsection{Root systems} We fix an $\bR$-root system $\Phi=(V,R,V^\vee, R^\vee)$. This means that $V, V^\vee$ are finite dimensional $\bR$-vector spaces, with a perfect bilinear pairing $(~,~): V\times V^\vee\to \bR$, $R\subset V\setminus\{0\},$ $R^\vee\subset V^\vee\setminus\{0\}$ are finite subsets in bijection
\begin{equation}
R\longleftrightarrow R^\vee,\ \al\longleftrightarrow\check\al,\ \text{such that }(\al,\check\al)=2.
\end{equation}
Moreover, the reflections
\begin{equation}
s_\al: V\to V,\ s_\al(v)=v-(v,\check\al)\al, \quad s_\al:V^\vee\to V^\vee,\ s_\al(v')=v'-(\al,v')\check\al, \quad \al\in R,
\end{equation}
leave $R$ and $R^\vee$ invariant, respectively. Let $W$ be the subgroup of $GL(V)$ (respectively $GL(V^\vee)$) generated by $\{s_\al:~\al\in R\}$.

We will assume that the root system $\Phi$ is reduced, meaning that
$\al\in R$ implies $2\al\notin R,$ and crystallographic, meaning that
$(\al,\check\al')\in\bZ$ for all $\al\in R,~\check\al'\in R^\vee.$ We
also assume that $R$ generates $V$.
We will fix a choice of simple roots $\Pi\subset R$, and consequently, positive roots $R^+$ and positive coroots $R^{\vee,+}.$ Often, we will write $\al>0$ or $\al<0$ in place of $\al\in R^+$ or $\al\in (-R^+)$, respectively. 

\smallskip

We fix a $W$-invariant inner product $\langle~,~\rangle$ on
$V$.  Denote also by $\langle~,~\rangle$ the dual inner product on  $V^\vee.$ If $v$ is a vector in $V$ or $V^\vee$, we denote $|v|:=\langle v,v\rangle^{1/2}.$


\subsection{The Clifford algebra}
A classical reference for the Clifford algebra is \cite{Ch} (see also
section II.6 in \cite{BW}). Denote by $ C(V)$ the Clifford algebra defined by $V$
and the inner product $\langle~,~\rangle$.  More precisely, $
C(V)$ is the
quotient of the tensor algebra of $V$ by the ideal generated by
$$\om\otimes \om'+\om'\otimes \om+2\langle \om,\om'\rangle,\quad
\om,\om'\in V.$$
Equivalently,  $ C(V)$ is the associative algebra
with unit generated by $V$ with relations:
\begin{equation}
\om^2=-\langle\om,\om\rangle,\quad \om\om'+\om'\om=-2\langle\om,\om'\rangle.
\end{equation}
Let $\mathsf{O}(V)$ denote the group of orthogonal transformation of
$V$ with respect to $\langle~,~\rangle$. This acts by algebra
automorphisms on $ C(V)$, and the action of $-1\in
\mathsf{O}(V)$ induces a grading
\begin{equation}
 C(V)= C(V)_{\mathsf{even}}+  C(V)_{\mathsf{odd}}.
\end{equation}
Let $\ep$ be the automorphism of $ C(V)$ which is $+1$ on $
C(V)_{\mathsf{even}}$ and $-1$ on $ C(V)_{\mathsf{odd}}$.
Let ${}^t$ be the transpose automorphism of $ C(V)$
characterized by
\begin{equation}
\om^t=-\om,\ \om\in V,\quad (ab)^t=b^ta^t,\ a,b\in C(V).
\end{equation}
The Pin group  is
\begin{equation}\label{pin}
\mathsf{Pin}(V)=\{a\in  C(V):~ \ep(a) V a^{-1}\subset
V,~ a^t=a^{-1}\}.
\end{equation}
It sits in a short exact sequence
\begin{equation}\label{ses}
1\longrightarrow \bZ/2\bZ \longrightarrow
\mathsf{Pin}(V)\xrightarrow{\ \ p\ \ } \mathsf{O}(V)\longrightarrow 1,
\end{equation}
where the projection $p$ is given by $p(a)(\om)=\ep(a)\om a^{-1}$.

\subsection{The spin modules $S$}\label{s:spinmod}
If $\dim V$ is even, the Clifford algebra $C(V)$ has a unique (up to equivalence) complex simple module
$(\gamma, S)$ of $ C(V),$ of dimension $2^{\dim
    V/2}$, endowed with a positive
definite Hermitian form $\langle ~,~\rangle_{ S}$ such that
\begin{equation}
\langle\gamma(a)s,s'\rangle_{ S}=\langle s,\gamma(a^t) s'\rangle_{ S},\quad\text{for all
  }a\in  C(V)\text{ and } s,s'\in  S. 
\end{equation}
When $\dim V$ is odd, there are two simple inequivalent unitary
modules $(\gamma_+,S^+),$ $(\gamma_-,S^-)$ of dimension  $2^{[\dim
  V/2]}$. 
In order to simplify the formulation of the results, we will often refer to any one of $S$, $S^+,$ $S^-$, as a spin module. When there is no possibility of confusion, we may also denote by $S$ any one of $S^+$ or $S^-$, in order to state results in a uniform way.

Via (\ref{pin}), a spin module $S$ is an irreducible unitary
$\mathsf{Pin}(V)$ representation.
It is well-known (e.g., section II.6 in \cite{BW}) that as $\mathsf{Pin}(V)$-representations, we have: 
\begin{equation}\label{Stens} 
S\otimes  S\cong \bigwedge^\bullet V, \text{ when }\dim
V \text{ is even},\quad S\otimes  S\cong \bigoplus_{i=0}^{[\dim V/2]}\bigwedge^{2i}
V,\text{ when }\dim V\text{ is odd.}
\end{equation}

\subsection{The spin cover $\wti W$}\label{Wspin}

The Weyl group $W$ acts by orthogonal
transformations on $V$, so one can embed $W$ as a subgroup of
$\mathsf{O}(V).$ We define the group $\wti W$ in
$\mathsf{Pin}(V)$: 
\begin{equation}
\wti W:=p^{-1}(\mathsf{O}(V))\subset \mathsf{Pin}(V),\text{ where $p$
  is as in (\ref{ses}).}
\end{equation}
Therefore, $\wti W$ is a central extension of $W$:
\begin{equation}
1\longrightarrow \bZ/2\bZ \longrightarrow
\wti W\xrightarrow{\ \ p\ \ } W\longrightarrow 1.
\end{equation}
The group $\wti W$ has a Coxeter presentation similar to that of $W$.
Recall that as a
Coxeter group, $W$ has a presentation:
\begin{equation}
W=\langle s_{\al},~\al\in\Pi|\  s_\al^2=1,~(s_\al
  s_\beta)^{m(\al,\beta)}=1, ~\al\neq\beta\in\Pi\rangle,
\end{equation}
for certain positive integers $m(\al,\beta).$
Theorem 3.2 in \cite{Mo} gives:
\begin{equation}
\wti W=\langle z,\wti s_{\al},~\al\in\Pi|\  z^2=1, \wti s_\al^2=z,~(\wti s_\al
  \wti s_\beta)^{m(\al,\beta)}=1, ~\al\neq\beta\in\Pi\rangle.
\end{equation}
We will also need the explicit embedding of $\wti W$ into
$\mathsf{Pin}(V).$

\begin{theorem}[{\cite[Theorem 3.2]{Mo}}]\label{embed}
The assignment
\begin{equation}
\begin{aligned}
&\tau(z)=-1\\
&\tau(\wti s_\al)=f_\al:=\al/|\al|,\quad \al\in\Pi,
\end{aligned}
\end{equation}
extends to a group homomorphism $\tau:\wti W\to \mathsf{Pin}(V).$ Moreover, we have
$\tau(\wti s_\beta)=f_\beta:=\beta/|\beta|,$ for all $\beta\in R^+$.
\end{theorem}

\begin{definition}\label{d:assoc}We call a representation $\wti\sigma$ of $\wti W$ genuine
(resp. non-genuine) if
$\wti\sigma(z)=-1$ (resp. $\wti\sigma(z)=1$). The non-genuine $\wti
W$-representations are the ones that factor through $W$.

We say that two genuine $\wti W$-types $\sigma_1,\sigma_2$ are associate if $\sigma_1\cong\sigma_2\otimes\sgn$.
\end{definition}

\subsection{}\label{s:2.5} Via the embedding $\tau$, we can regard $S$ if $\dim V$ is even (resp. $S^\pm$ if $\dim V$ is odd) as unitary (genuine)
$\wti W$-representations. Since $R$ spans $V$, they are irreducible representations (\cite[Theorem 3.3]{Mo}). When $\dim V$ is odd, $S^+$ and $S^-$ are associate in the sense of Definition \ref{d:assoc}, while if $\dim V$ is even, $S$ is self-associate. 

 Let $S$ be a spin $\wti W$-module and
  $(\sigma,U)$ a $W$-type. From (\ref{Stens}), we see that $\sigma\otimes S$ contains a spin $\wti W$-module
if and only if $\sigma$ appears as a constituent of $\bigwedge^\bullet V$. Moreover, it is known (\cite[Theorem 5.1.4]{GP}), that if
$W$ is irreducible, then $\bigwedge ^i V$,
  $0\le i\le \dim V$, forms a set of irreducible, pairwise inequivalent $W$-representations.

\subsection{The Casimir element of $\wti W$}\label{casimir} The notions in this
subsection are motivated by the results of \cite{BCT}, where the
element $\Omega_{\wti W}$ that we define here appeared naturally in the context of the Dirac operator for the graded affine Hecke algebra.

Let $c:R^+\to\bR$ be a $W$-invariant function.

\begin{definition}Denote
\begin{equation}\label{omWtilde}
\Omega_{\wti W,c}=\sum_{\substack{\al>0,\beta>0\\s_\al(\beta)<0}} c(\al) c(\beta)
 |\check\al| |\check\beta| ~\wti s_\al \wti s_\beta=\sum_{\substack{\al>0,\beta>0\\ \langle\al,\beta\rangle\neq 0}}
\frac{\langle\check\al,\check\beta\rangle}{|\cos(\al,\beta)|} c(\al)c(\beta) \wti s_\al \wti s_\beta.
\end{equation}
The equality holds because the contributions in the second sum of the pairs $\{\al,\beta\}$ and $\{s_\al(\beta),\al\}$ cancel out, whenever $s_\al(\beta)>0$. If $c\equiv 1,$ we write $\Omega_{\wti W}$ for $\Omega_{\wti W,1}.$
\end{definition}

If $C_w$ is the $W$-conjugacy class of $w\in W$, then there are two
possibilities for 
$p^{-1}(C_w)\subset \wti W$:

\begin{enumerate}
\item $p^{-1}(C_w)$ is a single $\wti W$-conjugacy class, or
\item $p^{-1}(C_w)$ splits into two conjugacy $\wti W$-classes $\wti C_w:=\{w':
  w'\in C_w\}$ and $z\wti C_{w}:=\{zw': w'\in C_w\}$.
\end{enumerate} 

One sees that if $w=s_\al s_\beta$, then the second case holds (\cite{Mo}).
This implies that we have
\begin{equation}
\Omega_{\wti W,c}\in \mathbb C[\wti W]^{\wti W}.
\end{equation}
In particular, every $\wti\sigma\in \widehat{\wti W}$ acts on
$\Omega_{\wti W,c}$ by a scalar, which we denote
$\wti\sigma(\Omega_{\wti W,c}).$

We will refer to $\Omega_{\wti W}$ as the Casimir element of $\wti W$. The justification for the name is given by Theorem
\ref{t:intro}(1). As a hint towards this result, let us recall  (e.g.,
\cite[p. 562]{Mo}) that 
\begin{equation}
\tr_S(s_\al s_\beta)=|\cos(\al,\beta)| \dim S,\ \al,\beta\in R^+,
\end{equation}
for a spin module $S$. This means that we have
\begin{equation}
S(\Omega_{\wti W})=\sum_{\al>0,\beta>0}\langle\check\al,\check\beta\rangle=\langle2\check\rho,2\check\rho\rangle,
\end{equation}
where $\check\rho=\frac 12\sum_{\al>0}\check\al.$

\section{$\wti W$-types}\label{s:3}

In this section, we prove our main results, Theorems \ref{t:intro} and \ref{t:gen}. Before that,  we recall certain elements from the theory of elliptic representations of a finite group. While these elements are not necessary for proving Theorems \ref{t:intro} and \ref{t:gen}, they are useful for setting our result in the appropriate context.

For a finite group $\Gamma$, let $\C R(\Gamma)$ denote the representation theory ring of $\Gamma,$ and let $\widehat \Gamma$ denote the set of irreducible representations of $\Gamma.$

\subsection{Elliptic representations of a finite group}\label{s:ell}
The reference for most of the results in 
this and the next subsection is \cite{R}. Assume first that $\Gamma$ is an arbitrary finite subgroup of $\mathsf{GL}(V)$. An element $\gamma\in\Gamma$ is called elliptic (or anisotropic) if $V^\gamma=0$. Let $\Gamma_{\mathsf{ell}}$ denote the set of elliptic elements in $\Gamma$. This is closed under conjugation by $\Gamma.$ Let $\mathcal L$ be the set of subgroups $L\subseteq\Gamma,$ such that $V^L\neq 0$. For every $L\in \mathcal L$, let $\Ind_L^\Gamma:\C R(L)\to \C R(\Gamma)$ be the induction map, and denote
\begin{align}
&\C R_{\mathsf{ind}}(\Gamma)=\sum_{L\in \C L}\Ind_L^\Gamma(\C R(L))\subseteq \C R(\Gamma),\\
&\overline{\C R}(\Gamma)=\C R(\Gamma)/\C R_{\mathsf{ind}}(\Gamma).
\end{align}
One calls $\overline{\C R}(\Gamma)$ the space of elliptic representations of $\Gamma.$

Define a bilinear pairing, called the elliptic pairing on $\Gamma$:
\begin{align}
e_\Gamma(\sigma,\sigma')&=\sum_{i\ge 0} (-1)^i \dim\operatorname{Hom}_\Gamma(\bigwedge^i
V\otimes \sigma,\sigma'),\quad \sigma,\sigma'\in \C R(\Gamma).
\end{align}

Proposition 2.2.2 in \cite{R} shows, in particular, that the radical
of $e_\Gamma$ is precisely $\C R_{\mathsf{ind}}(\Gamma)$, and thus
$e_\Gamma$ induces a nondegenerate bilinear form on $\overline{\C
  R}(\Gamma)$.  Moreover, if $\CC_{\mathsf{ell}}(\Gamma)$ denotes the set of $\Gamma$-conjugacy classes in $\Gamma_{\mathsf{ell}},$ we have 
\begin{equation}\label{dimell}
\dim\overline{\C R}(\Gamma)=|\CC_{\mathsf{ell}}(\Gamma)|.
\end{equation}

\begin{lemma}\label{l:red} If $\sigma\in \C R(\Gamma)$, then
  $\sigma\otimes \bigwedge^{\dim V}V-(-1)^{\dim V}\sigma$ is in $\C R_{\mathsf{ind}}(\Gamma)$. 
\end{lemma}

\begin{proof}
We have $\bigwedge^i V\otimes\bigwedge^{\dim V} V\cong\bigwedge^{\dim V-i}
V$, as $\Gamma$-representations, for all $0\le i\le \dim V.$ From
this, it follows that $e_\Gamma(\sigma\otimes \bigwedge^{\dim
  V}V,\sigma')=(-1)^{\dim V} e_\Gamma(\sigma,\sigma'),$ for all
$\sigma,\sigma'$.
\end{proof}

Let $\C R^{\mathsf{red}}(\Gamma)$ denote the quotient of $\C
R(\Gamma)$ by the subspace generated by $\sigma\otimes \bigwedge^{\dim
  V}V-(-1)^{\dim V}\sigma$, for all $\Gamma$-types $\sigma.$ Lemma
\ref{l:red} implies that the natural (surjective) map
\begin{equation}\label{mapred}
\C R^{\mathsf{red}}(\Gamma)\to \overline{\C R}(\Gamma)
\end{equation}
is well-defined and preserves $e_\Gamma.$

\subsection{Elliptic representations of $W$}\label{ellW}
We specialize to $\Gamma=W$ here acting on the reflection representation $V$; \cite{R} analyzes the relation between $\overline{\C
  R}(W)$ and Springer representations. 

Let $\fg$ be the complex Lie algebra determined by the root system $\Phi$, and
let $G$ be the simply connected connected Lie group with Lie algebra
$\fg.$ For every $x\in \fg,$ let $Z_G(x)$ denote the centralizer of
$x$ in $G$, and let $Z_G(x)^0$ be the identity component. Define the
A-group of $x$ in $G$ to be the quotient
\begin{equation}
A(x)=Z_G(x)/Z_G(x)^0Z(G),
\end{equation}
where $Z(G)$ is the center of $G$.

Specialize to the case when $x\in\C N(\fg),$ the set of nilpotent elements of $\fg.$ By
the Jacobson-Morozov theorem, there exists a Lie algebra homomorphism
$\kappa: sl(2,\bC)\to \fg$ such that
$\kappa\left(\begin{matrix}0&1\\0&0\end{matrix}\right)=e.$
Let $\fk s_0$ be a the semisimple part of the Lie algebra of a maximal torus in $Z_G(\kappa(sl(2,\bC))$. As explained in
  \S 3.2 of \cite{R}, the group $A(e)$ acts naturally on $\fk s_0$,
  i.e., $A(e)\subset \mathsf {GL}(\fk s_0)$,
  and therefore we may define $\overline {\C R}(A(e))$ with respect to
  this action.

\begin{definition}\label{quasi} 
An element $e\in \C N(\fg)$ is called distinguished if $Z_G(e)$ contains no nontrivial
torus.
An element $e\in\C N(\fg)$ is called quasi-distinguished if there
exists a semisimple element $t\in Z_G(e)$ such that $t\exp(e)$ centralizes
no nontrivial torus in $G$.
In particular, every distinguished $e$ is quasi-distinguished with $t=1$.
\end{definition}

Corollary 3.2.3 in \cite{R} shows that $\overline {\C  R}(A(e))\neq 0$
if and only if $e$ is quasi-distinguished in $\fg.$

From Springer theory, recall $\C B_e$, the variety of Borel
subalgebras of $\fg$ containing $e$. Let $d_e$ denote the complex
dimension of $\C B_e.$ Since $Z_G(e)$ acts on $\C B_e$,
we get an action of $A(e)$ on the cohomology $H^\bullet(\C B_e)$. Let
$\widehat {A(e)}_0$ be the set of irreducible $A(e)$-representations
that appear in this action, and let  
$\C R_0(A(e))$ be the subspace of $\C R(A(e))$ spanned by $\widehat {A(e)}_0$.

The Springer correspondence constructs an action of $W$ on
$H^\bullet(\C B_e)$ which commutes with the action of $A(e)$, and
gives a map:
\begin{equation}
~G\backslash\{(e,\phi): e\in\C N(\fg),~\phi\in\widehat {A(e)}_0\}\longrightarrow \widehat
  W,\quad  (e,\phi)\mapsto\sigma_{(e,\phi)}:=\Hom_{A(e)}[\phi,
    H^{2d_e}(\C B_e)],
\end{equation}
which is well-defined and bijective.

For every $\phi\in \widehat {A(e)}_0$, set 
\begin{equation}
H_e(\phi):=\Hom_{A(e)}[\phi,H^\bullet(\C B_e)],
\end{equation}
and regard it as an element of $\C R(W)$. Let $\C R_e(W)$ be the span
in $\C R(W)$ of $H_e(\phi)$, $\phi\in \widehat{A(e)}_0$. The space
$\{H_e(\phi): e,\phi\}$ is basis of $\C R(W)$, and therefore, we have
a decomposition $\C R(W)=\sum_e \C R_e(W)$, which induces a
decomposition $\overline{\C R}(W)=\sum_e\overline{\C R}_e(W).$

\begin{theorem}[\cite{R}]\label{t:reeder}
The map $H_e:\C R_0(A(e))\to \C R_e(W)$ induces a vector space
isomorphism $\overline H_e: \overline{\C R}_0(A(e))\to \overline{\C
  R}_e(W)$. Moreover, we have:
\begin{enumerate}
\item The isomorphism $\overline H_e$ is an isometry with respect to
  the elliptic pairings $e_W$ and $e_{A(e)}$;
\item The spaces $\overline{\C R}_0(A(e))$ and $ \overline{\C
  R}_e(W)$ are nonzero if and only if $e$ is quasi-distinguished;
\item If $e$ is distinguished, then $\{H_e(\phi):~\phi\in
  \widehat{A(e)}_0\}$ is an orthonormal basis of $\overline{\C
    R}_e(W)$ with respect to $e_W$.
\end{enumerate}
\end{theorem}

\subsection{Elliptic representations of $\wti W$}\label{ellWtil} We specialize now to
$\Gamma=\wti W$ acting also on the (nongenuine) reflection representation $V$. 
Let $\C R_{\mathsf{gen}}(\wti W)$ denote the subspace of $\C R(\wti
W)$ spanned by $\widehat{\wti W}_{\mathsf{gen}}$,
the irreducible genuine $\wti W$-types. Every nongenuine $\wti W$-type
is a pullback of a $W$-type, so we may regard $\C R(W)$ naturally as a
subspace of $\C R(\wti W)$. Clearly, we have
\begin{equation}
e_{\wti W}(\wti\sigma,\sigma')=0, \text{ if }\wti\sigma\in \C
R_{\mathsf{gen}}(\wti
W), \ \sigma'\in \C R(W),
\end{equation} 
therefore we have an orthogonal decomposition $\C R(\wti W)=\C
R_{\mathsf{gen}}(\wti W)\oplus \C R(W).$ As before, define $\overline {\C R}(\wti W)$  to be
the quotient by the radical of $e_{\wti W}$, and let $\overline {\C
  R}_{\mathsf{gen}}(\wti W)$ be the image of $\C R_{\mathsf{gen}}(\wti W)$ in $\overline {\C
  R}(\wti W)$. Consequently, there is an orthogonal decomposition
\begin{equation}
\overline {\C R}(\wti W)=\overline {\C R}_{\mathsf{gen}}(\wti W)\oplus \overline
          {\C R}(W).
\end{equation}
From (\ref{dimell}), we have $\dim\overline {\C R}_{\mathsf{gen}}(\wti W)=|\CC_{\mathsf{ell}}(\wti W)|-|\CC_{\mathsf{ell}}(W)|.$
Recall the projection $p:\wti W\to W$ from (\ref{ses}). Since $V$ is a nongenuine representation of $\wti W$, an element $\wti w\in \wti W$ is elliptic if and only if $p(\wti w)\in W$ is elliptic. Recall that if $C_w\subset W$ is a conjugacy class, then  $p^{-1}(C_w)$ is a single conjugacy class in $\wti W$ or it splits into two conjugacy classes in $\wti W.$ Let $\CC^0(W)$ denote the set of conjugacy classes of $W$ which split in $\wti W$ and set $\CC^0_{\mathsf{ell}}(W)=\CC^0(W)\cap \CC_{\mathsf{ell}}(W).$ Then we have
\begin{equation}\label{compare}
\dim \overline {\C R}_{\mathsf{gen}}(\wti W)=|\CC^0_{\mathsf{ell}}(W)|\le |\CC_{\mathsf{ell}}(W)|=\dim \overline {\C R}(W).
\end{equation}

The dimension of the second space is as follows, see \cite{Ca},\cite[\S3.1]{R}:
\begin{equation}\label{Rlist}
\begin{aligned}
&A_{n-1}: \quad 1, \\
&B_n: \quad\text{ the number of partitions of }n,\\ 
&D_n: \quad\text{the number of partitions of $n$ with even number of parts},\\
&G_2: \ 3,\quad F_4:\ 9,\quad E_6:\ 5,\quad E_7:\ 12,\quad E_8:\ 30.
\end{aligned}
\end{equation}
Since $\wti \sigma(z)=-1$ for every genuine $\wti W$-type $\wti
\sigma$, it is clear that if $C_w\notin \CC^0(W)$, then
$\tr_{\wti\sigma}(\wti w)=0,$ for all $\wti w\in p^{-1}(C_w)$ and all
genuine $\wti W$-types $\wti\sigma.$

\begin{lemma}\label{l:charspin} Let $S$ be a spin module for $\wti W.$ 
\begin{enumerate}
\item If $\dim V$ is even, then $\tr_S(\wti w)\neq 0$ if and only if
  $\det_V(1+p(\wti w))\neq 0.$
\item If $\dim V$ is odd, then $\tr_S(\wti w)\neq 0$ if 
  $\det_V(1-p(\wti w))\neq 0$ (i.e., if $\wti w$ is elliptic). In
  particular, $\CC^0_{\mathsf{ell}}(W)=\CC_{\mathsf{ell}}(W)$ in this case.
\end{enumerate}
\end{lemma}

\begin{proof}If $\dim V$ is even, and $S$ is
the spin module, we see by (\ref{Stens}) that $\tr_{S}(\wti
w)^2=\tr_{\bigwedge^\bullet V}(p(\wti w))=\det_V(1+p(\wti w))$, and
this proves (1). If $\dim V$ is odd, and $S^+,S^-$ are the two spin
modules, (\ref{Stens}) implies that $\tr_{S^+-S^-}(\wti
w)^2=2\oplus_{i=0}^{\dim V}(-1)^i\tr_{\bigwedge^i V}(p(\wti
w))=\det_V(1-p(\wti w))$. Since $S^+$ and $S^-$ are associate,
$\tr_{S^+-S^-}(\wti w)\neq 0$ implies $\tr_{S^\pm}(\wti w)\neq 0$, and
 this proves (2).
\end{proof}

\begin{remark}\label{r:ellclasses}
In type $A_{n-1}$, there is a single elliptic conjugacy class,
consisting of the $n$-cycles in $S_n$, and it is easy to check
directly that it splits in $\wti S_n.$ Theorem
4.1 and Lemma 6.4 in \cite{Re} for $B_n$ and $D_n$ respectively, and
sections 6--9 in \cite{Mo2} for the exceptional groups, show that if $\dim V$
is even, the
split elliptic conjugacy classes are precisely the ones on which $S$
does not vanish.
In terms of the classification of elliptic elements of $W$ from Carter
\cite{Ca}, when $\dim V$ is even, the set $\CC^0_{\mathsf{ell}}(W)\neq \CC_{\mathsf{ell}}(W)$ is explicitly as follows:
\begin{enumerate}
\item in type $B_{2n}$, the elliptic conjugacy classes corresponding to partitions of $2n$ with only even parts;
\item in type $D_{2n}$, the elliptic conjugacy classes corresponding
  to partitions of $2n$ with only even parts, or partitions with only
  odd parts and multiplicity one;
\item in $G_2$: $\{A_2,G_2\}$;
\item in $F_4$: $\{A_2+\wti A_2, D_4(a_1), B_4, F_4, F_4(a_1)\}$;
\item in $E_6$: $\{3A_2,E_6,E_6(a_1),E_6(a_3)\}$;
\item in $E_8$: $\{A_8,2A_4,4A_2,D_8(a_1),D_8(a_2),2D_4(a_1),E_6(a_2)+A_2,E_6+A_2,E_8,\\E_8(a_1),E_8(a_2),E_8(a_3),E_8(a_4),E_8(a_5),E_8(a_6),E_8(a_7),E_8(a_8)\}.$
\end{enumerate}
\end{remark}
Let $S$ denote a spin $\wti W$-module. One can consider the linear map: 
\begin{equation}
\iota_{\C S}: \C R(W)\longrightarrow \C R_{\mathsf{gen}}(\wti W),\quad
\iota_{\C S}(\sigma)=\sigma\otimes S.
\end{equation}
Since the genuine $\wti W$-types are determined by their values on
$p^{-1}(\CC^0(W))$, the map $\iota_S$ is
surjective if and only if $\tr_S$ does not vanish on any conjugacy
class in $p^{-1}(\CC^0(W))$. Using Lemma \ref{l:charspin} and \cite{Re,Mo2} again, we see that:

\begin{lemma}\label{l:surj}
The map $\iota_S$ is surjective if and only if $W$ is of type $B_n,$ $D_n$, $G_2$, $F_4$, or $E_8.$
\end{lemma}
Since $\C R_{\mathsf{ind}}(W)$ (resp. $\C R_{\mathsf{gen,ind}}(\wti W)$) can be identified with the vector subspace of virtual characters that vanish on $W_{\mathsf{ell}}$ (resp. $\wti W_{\mathsf{ell}}$), we have
\begin{equation}
\iota_S(\C R_{\mathsf{ind}}(W))\subset \C R_{\mathsf{gen,ind}}(\wti W),
\end{equation}
 and so $\iota_S$ gives a linear map
\begin{equation}
\overline\iota_S:\overline{\C R}(W)\to \overline{\C R}_{\mathsf{gen}}(\wti W).
\end{equation}

\begin{proposition}\label{l:surjbar}
The map $\overline\iota_S$ is surjective.
\end{proposition}

\begin{proof}
By Lemma \ref{l:surj}, this is true when $\wti W$ is of type
$B_n,D_n,G_2,F_4,E_8$, since $\iota_S$ is surjective. The conclusion
is implied if $\tr_S$ is nonzero on every conjugacy class in
$p^{-1}(\CC^0_{\mathsf{ell}}(W))$; this turns out to be the case for
every irreducible $W$, by Lemma \ref{l:charspin} and Remark \ref{r:ellclasses}.
\end{proof}

\begin{corollary}
The map $\overline\iota_S$ is an isomorphism if and only if $\dim V$ is odd or $W$ is of type $A.$
\end{corollary}

\begin{proof}
This is immediate from Proposition \ref{l:surjbar} by comparing the
dimension of the two spaces as in (\ref{compare}) and Lemma \ref{l:charspin}.
\end{proof}

\subsection{}
 Applying (\ref{mapred}) to this setting, we get a surjective linear
 map $ \C R_{\mathsf{gen}}^{\mathsf{red}}(\wti W)\to \overline{\C R}_{\mathsf{gen}}(\wti
W)$ which preserves the elliptic pairing.
Thus we have constructed two maps:
\begin{equation}\label{chain}
\overline{\C R}(W)\twoheadrightarrow\overline{\C R}_{\mathsf{gen}}(\wti
W)\twoheadleftarrow \C R_{\mathsf{gen}}^{\mathsf{red}}(\wti W).
\end{equation}
Via 
\begin{equation}
\xi: \C R_{\mathsf{gen}}(\wti W)\to  \C R_{\mathsf{gen}}(\wti W), \quad \wti\sigma\mapsto \wti\sigma\otimes\sgn+(-1)^{\dim V}\wti\sigma, 
\end{equation}
we may identify $ \C R_{\mathsf{gen}}^{\mathsf{red}}(\wti W)$ with the
image of $\xi$, i.e., with the subspace of $\C
R_{\mathsf{gen}}(\wti W)$ spanned by $\{\wti\sigma\otimes\sgn+(-1)^{\dim
  V}\wti\sigma\neq 0: \wti\sigma\in \widehat{\wti W}/_\sim\}$ (here we think of
$\widehat{\wti W}/_\sim$ as a system of representatives for the
symmetry classes). This shows that $\dim   \C
R_{\mathsf{gen}}^{\mathsf{red}}(\wti W)=|\widehat{\wti W}/_\sim|$ when
$\dim V$ is even, and $|\widehat{\wti W}/_\sim|-|\{\wti\sigma\in
\widehat{\wti W}: \wti\sigma\otimes\sgn\cong\wti\sigma\}|$ when $\dim V$ is
odd. From \cite{Mo2,Re,Sc}, we see that the dimension of
$|\widehat{\wti W}/_\sim|$ equals:
\begin{equation}\label{dimred}
\begin{aligned}
&A_{n-1}: \text{ the number of partitions of $n$ into distinct
  parts},\\
&B_n:\text{ the number of partitions of $n$},\\
&D_n:\text{ the number of equivalence classes of partitions of $n$
  under transposition},\\
&G_2: \ 3,\quad F_4:\ 9,\quad E_6:\ 6,\quad E_7:\ 13,\quad E_8:\ 30.
\end{aligned}
\end{equation}
In addition, in types $B_{2n+1},D_{2n+1},E_7$, there are no
self-associate $\wti W$-types. In type $A_{2n-1}$, the number of
self-associate $\wti S_{2n}$-types equals the number of partitions of odd
length of $2n$ into distinct parts.   
Comparing (\ref{dimred}) with (\ref{Rlist}), we see  that if $W$ is of type $B_n$, 
  $G_2$, $F_4$, or $E_8$, then we have $\dim \overline{\C R}(W) =\dim
  \C R_{\mathsf{gen}}^{\mathsf{red}}(\wti W)$. (If the type is
  $B_{2n+1}$, then the maps in (\ref{chain}) are both
  isomorphisms.)
\begin{remark}
One can ask if there is a natural linear map
   $\overline{\C R}(W)\to \C
  R^{\mathsf{red}}_{\mathsf{gen}}(\wti W)$, having good properties with respect to
the elliptic pairing. When $\dim V$ is odd, it is easy to check that
such a map is $\sigma\mapsto \frac 1{\sqrt 2}\sigma\otimes (S^+-S^-)$, and that this
map is an injective metric with respect to the elliptic pairing on $ \overline{\C
  R}(W)$ and the standard pairing on $\C
  R_{\mathsf{gen}}(\wti W)$. When $\dim V$ is even, a
  similar construction exists, but instead of $\wti W$, one needs
  to consider $\wti W_{\mathsf{even}}=\{\wti w\in \wti W: \sgn(\wti
    w)=1\}.$ This fits naturally with the theory of the Dirac index
    for graded Hecke algebras and it is considered in \cite{CT}.
\end{remark}

\subsection{The classification of $\wti W$-types} \label{main} The
rest of the section is dedicated to the proof of Theorem
\ref{t:intro}. 

By the Jacobson-Morozov theorem, we know that there is
a one-to-one correspondence between $G$-orbits of nilpotent elements
in $\fg$, and the set $\C T(G)$ of $G$-conjugacy classes of Lie
triples in $\fg$:
\begin{equation}
(e,h,f):\  [h,e]=2e,
  [h,f]=-2f, [e,f]=h.
\end{equation}

\begin{definition}\label{d:solv} 
 Let $\C N_0(\fg)$ denote the set of all nilpotent elements $e$ whose
 centralizer in $\fg$ is a solvable subalgebra. Let $\C T_0(G)$ denote the
 set of $G$-conjugacy classes of triples $(e,h,f)$ such that $e\in \C N_0(\fg).$
\end{definition}

This definition of $\C T_0(G)$ agrees with the one from the
introduction by Proposition 2.4 in \cite{BV}. 
Every quasi-distinguished
 nilpotent element is in $\C N_0(\fg)$, but in types
 $A,D,E_6$, not all $e\in \C N_0(\fg)$ are quasi-distinguished in the sense of Definition \ref{quasi}. For
 example, in
 $sl(n)$, the only quasi-distinguished nilpotent orbit is the regular
 orbit, but $\C N_0(\fg)$ contains every orbit whose Jordan form has all blocks of
distinct sizes.

\subsection{Type $A$}We begin the proof of Theorem \ref{t:intro}. This is a case-by-case verification, combinatorially for classical root systems, and a direct computation for exceptional.

 The starting point is type $A_{n-1}$, $\wti W_n=\wti S_n.$ Let $P(n)$
 be the set of all partitions of $n$, and let $DP(n)$ be the set of
 distinct partitions. If
 $\lambda=(\lambda_1,\lambda_2,\dots,\lambda_m)$ is a partition of
 $n$, written in decreasing order, we denote the length of $\lambda$
 by $\ell(\lambda)=m$. We say that $\lambda$ is even (resp. odd) if
 $n-\ell(\lambda)$ is even (resp. odd). 

It is well-known that every partition $\lambda$ parameterizes a unique
$S_n$-type $\sigma_\lambda$, and this gives a one-to-one
correspondence between $P(n)$ and $\widehat S_n$. The first part of Theorem \ref{t:intro} for $\wti S_n$ is a classical result of Schur.

\begin{theorem}[\cite{Sc}] There exists a one-to-one correspondence 
$$\widehat{\wti S_n}/_\sim\longleftrightarrow DP(n).$$ For every
  even $\lambda\in DP(n)$, there exists a unique
  $\wti\sigma_\lambda\in \widehat{\wti S_n}$, and for every odd
  $\lambda\in DP(n)$, there exist two associate $
  \wti\sigma_\lambda^+,\wti\sigma_\lambda^-\in \widehat{\wti
  S_n}$. The dimension of $\wti\sigma_\lambda$
  or $\wti\sigma_\lambda^\pm$ is
\begin{equation}
2^{[\frac{n-m}2]}\frac{n!}{\lambda_1!\dots\lambda_m!}\prod_{1\le
  i<j\le m}\frac{\lambda_i-\lambda_j}{\lambda_i+\lambda_j}.
\end{equation}
\end{theorem}

Notice that $DP(n)$ precisely parameterizes the set of
quasi-distinguished orbits in type $A_{n-1}$ (the only local systems
of Springer type are the trivial ones here).

In order to simplify the
formulas below, we write
$\wti\sigma_\lambda:=\wti\sigma_\lambda^+\oplus\wti\sigma_\lambda^-$,
if $\lambda$ is an odd partition in $DP(n)$.

The decomposition of the tensor product of an $S_n$-type
$\sigma_\mu$ with the spin representation $S=\wti\sigma_{(n)}$ is
well-known (see \cite[\S9.3]{St} for example). If $\lambda\neq (n)$
(this case has been covered by \S\ref{s:2.5} already), we have:
\begin{equation}
\dim\Hom_{\wti
  S_n}[\wti\sigma_\lambda,\sigma_\mu\otimes\wti\sigma_{(n)}]=\frac
1{\epsilon_\lambda\epsilon_{(n)}} 2^{\frac{\ell(\lambda)-1}2} g_{\lambda,\mu},
\end{equation}
where $\ep_\lambda=1$ (resp. $\ep_\lambda=\sqrt 2$) if $\lambda$ is
even (resp. odd), and $g_{\lambda,\mu}$ are certain Kostka type
numbers (\cite[\S9.3]{St}). In particular,
$g_{\lambda,\lambda}\neq 0$, and this proves  (2) in Theorem \ref{t:intro}.

In order to verify claim (1) of Theorem \ref{t:intro}, we need a formula for the character of
$\wti\sigma_\lambda$ on the conjugacy class represented by the cycle
$(123)$ in $\wti S_n$. This may be well-known, but I could not
find a reference, so I include a combinatorial proof.

\begin{lemma}\label{l:SnA2} If $\lambda=(\lambda_1,\dots,\lambda_m)$ is a partition
  in $DP(n)$, $n\ge 3$, then we have
\begin{equation}
|C^{S_n}_{(123)}|\frac
{\tr_{\wti\sigma_{\lambda}}((123))}{\dim\wti\sigma_{\lambda}}=\sum_{i=1}^m\frac{\lambda_i(\lambda_i^2-1)}6-{n\choose
2},
\end{equation}
where $C^n_{(123)}$ denotes the conjugacy class of the cycle $(123)$
in $\wti S_n$.
\end{lemma}

\begin{proof}
The proof is by induction on $n$. One can immediately verify this for
$n=3$, and let us assume it holds for $n-1$. By restriction to $\wti
S_{n-1}$, one has
$$\tr_{\wti\sigma_{\lambda}}=\sum_{i=1}^m\tr_{\wti
  \sigma_{\lambda^i}},$$
where $\lambda^i=(\lambda_1,\dots,\lambda_i-1,\dots,\lambda_m)$. One
discards $\lambda^i$ if it is not in $DP(n-1)$. Using the induction
hypothesis, and the ratios $|C^{S_n}_{(123)}|/|C^{S_{n-1}}_{(123)}|=\frac
n{n-3}$, $\dim{\wti\sigma_{\lambda^i}}/\dim{\wti\sigma_\lambda}=\frac
{\lambda_i}n\prod_{j\neq
  i}\frac{(\lambda_i-\lambda_j-1)(\lambda_i+\lambda_j)}{(\lambda_i-\lambda_j)(\lambda_i+\lambda_j-1)}$,
an elementary calculation leads to  the following
curious identity that we need to verify:
\begin{equation}\label{eq:iden} 
\sum_{i=1}^m\lambda_i^2(\lambda_i-1)\prod_{j\neq
  i}\frac{(\lambda_i-\lambda_j-1)(\lambda_i+\lambda_j)}{(\lambda_i-\lambda_j)(\lambda_i+\lambda_j-1)}
=\sum_{i=1}^m \lambda_i^2(\lambda_i-1)-\sum_{i\neq j}\lambda_i\lambda_j.
\end{equation}
A similar identity arises when one proves Schur's dimension formula
by induction (\cite[Proposition 10.4]{HH}). We consider the
function
\begin{equation}\label{fx1}
f(x)=(x^2-x)\prod_{i=1}^m \frac{(x-\lambda_i-1)(x+\lambda_i)}{(x-\lambda_i)(x+\lambda_i-1)},
\end{equation} 
which has the expansion 
\begin{equation}\label{fx2}
f(x)=(x^2-x)+\sum_{i=1}^m A_i-\sum_{i=1}^n \frac{A_i \lambda_i(\lambda_i-1)}{(x-\lambda_i)(x+\lambda_i-1)},
\end{equation}
where $A_i=-2\lambda_i\prod_{j\neq
  i}\frac{(\lambda_i-\lambda_j-1)(\lambda_i+\lambda_j)}{(\lambda_i-\lambda_j)(\lambda_i+\lambda_j-1)}.$
Notice that the coefficient of $x^{-2}$ in the Laurent expansion of
the right hand side of (\ref{fx2}) is precisely $(-2)$ times the left
hand side of (\ref{eq:iden}). Then one verifies (\ref{eq:iden}) easily,
by computing the coefficient of $x^{-2}$ in the Laurent expansion of (\ref{fx1}).
\end{proof}

Using the formula in Lemma \ref{l:SnA2}, we immediately check that
$\wti\sigma_\lambda(\Omega_{\wti
  S_n})=\sum_{i=1}^n\frac{\lambda_i(\lambda_i^2-1)}{3}.$ Here, for
simplicity, we assumed that the roots of type $A_{n-1}$ are the
standard ones. A middle element $h$ of a Lie triple for the nilpotent orbit indexed by the
partition $\lambda$ is, in coordinates,
$\left(-{(\lambda_1-1)},\dots,{(\lambda_1-1)},\dots,
-{(\lambda_m-1)},\dots,{(\lambda_m-1)}\right)$, and now claim (1) in
Theorem \ref{t:intro} is established.

\subsection{Types $B,C$} The nilpotent orbits in $sp(2n)$ and $so(m)$ are parameterized (via an analogue of the Jordan canonical form) by partitions of $2n$ (resp. $m$), where the odd (resp. even) parts occur with even multiplicity. Such an orbit is in $\C N_0(sp(2n))$ (resp. $\C N_0(so(m))$ if and only if the associated partition has only even (resp. odd) parts, and all parts have multiplicity at most 2. 

Let $W_n$ denote the Weyl group of type $B_n/C_n$. The group $W_n$ is
a semidirect product $W_n=S_n\rtimes (\bZ/2\bZ)^n$, and therefore, as
it is well-known, its representations are obtained by Mackey
theory. More precisely, let
$\chi_k=(\triv\boxtimes\dots\boxtimes\triv)^{n-k}\boxtimes
(\sgn\boxtimes\dots\boxtimes\sgn)^k$ be a character of $(\bZ/2\bZ)^n$,
and let $S_{n-k}\times S_k$ be the isotropy group of $\chi_k$ in
$S_n$. For every partitions $\lambda$ of $n-k$ and $\mu$ of $k$, one
constructs an irreducible $W_n$ representation
$\sigma_{(\lambda,\mu)}$ as 
\begin{equation}\label{indBn}
\sigma_{(\lambda,\mu)}=\Ind_{S_{n-k}\times S_k\times
  (\bZ/2\bZ)^n}^{W_n}(\sigma_\lambda\boxtimes \sigma_\mu\boxtimes \chi_k).
\end{equation}
This gives a bijection
\begin{equation}
\widehat W_n\longleftrightarrow BP(n),\quad
\sigma_{(\lambda,\mu)}\leftrightarrow (\lambda,\mu),
\end{equation}
where $BP(n)$ is the set of bipartitions of $n$.
In particular, in this notation, if $\lambda\in P(n)$, the representation
$\sigma_{(\lambda,\emptyset)}$ is obtained from the $S_n$-type
$\sigma_{\lambda}$, by letting the simple reflections of $W_n$ not in $S_n$
act by the identity.

Let $\wti W_n$ denote the spin cover of $W_n$. The genuine representations of $\wti W_n$
were classified by \cite{Re}, starting with the classification for
$\wti S_n$-types.

\begin{theorem}[{\cite[Theorem 5.1]{Re}}]\label{t:Bn} There is a one-to-one
  correspondence 
$$\widehat{(\wti W_n)}_{\mathsf{gen}}/_\sim\longleftrightarrow P(n).$$
For every $\lambda\in P(n)$, there exist:

\begin{enumerate}
\item one irreducible $\wti W_n$-type $\wti
  \sigma_{\lambda}=\sigma_{(\lambda,\emptyset)}\otimes S,$ if $n$ is even;
\item two associate $\wti W_n$-types $\wti
  \sigma_{\lambda}^{\pm}=\sigma_{(\lambda,\emptyset)}\otimes S^{\pm},$ if $n$ is odd.
\end{enumerate} 
\end{theorem} 

This realization is very convenient for computing the character of
$\wti\sigma_\lambda$ (and similarly $\wti\sigma_\lambda^\pm$). Using the character of $S$ and the characters of
type $A_{n-1}$, we immediately see that:
\begin{equation}
\frac{\tr_{\wti\sigma_\lambda}(\wti s_\al
  \wti s_\beta)}{\dim{\wti\sigma_\lambda}}=|\cos(\al,\beta)|\left\{\begin{matrix}\frac{\tr_{\sigma_\lambda}((123))}{\dim\sigma_\lambda},&\text{
  if }\al,\beta\text{ form an }A_2\\ \frac{\tr_{\sigma_\lambda}((12))}{\dim\sigma_\lambda},&\text{
  if }\al,\beta\text{ form an }B_2/C_2\end{matrix}\right..
\end{equation}
The relevant formulas for $\tr\sigma_\lambda$ in $S_n$ go back to Frobenius. In the
form that we need, they are:
\begin{equation}
|C_{(123)}^{S_n}|\frac
{\tr_{\sigma_\lambda}((123))}{\dim\sigma_\lambda}=p_2(\lambda)-{n\choose
  2},\quad |C_{(12)}^{S_n}|\frac
{\tr_{\sigma_\lambda}((12))}{\dim\sigma_\lambda}=p_1(\lambda),
\end{equation}
where $p_k(\lambda)$ is the sum of $k$ powers of the content of
$\lambda$.

The formula for $\wti\sigma_\lambda(\Omega_{\wti W})$ becomes
very explicit. Assume for simplicity of notation that we take the
standard Bourbaki coordinates for roots for type $B$ and $C$. Set 
\begin{equation}
\epsilon=1,\text{ if }\Phi=B_n, \quad \epsilon=\frac 12,\text{ if } \Phi=C_n.
\end{equation}
Then, we have
\begin{equation}
\wti\sigma_\lambda(\Omega_{\wti W})=4p_2(\lambda+\epsilon),
\end{equation}
where $p_2(\lambda+\epsilon)$ denotes the sum of squares of the
$\epsilon$-content of $\lambda$, i.e., the content of the $(i,j)$-box
is $j-i+\epsilon$.

The set of all contents of $\lambda+\epsilon$ (with repetitions)
represents the coordinates of one half of 
the middle element of a nilpotent orbit $\C O_{\lambda,\ep}$
in
$so(2n+1)$ (if $\ep=1$), respectively $sp(2n)$ (if $\ep=1/2$). Moreover, the nilpotent orbit lies in fact in $\C N_0(\fg).$ (See \cite[\S4.4]{CK2} for the combinatorics needed to verify this claim.)

The following algorithm (due to Slooten) attaches to each partition $\lambda$ with
content $\lambda+\epsilon$ a set of bipartitions. All the $W_n$-types parameterized by these bipartitions are Springer representations for the nilpotent orbit $\C O_{\lambda,\ep}$. This can be seen by comparing Slooten's algorithm with Lusztig's algorithm with S-symbols for the Springer correspondence (\cite[\S4.4]{CK2}). Moreover, the relation between this
algorithm and the elliptic representation theory of $W$ (via the elliptic representation theory of the graded affine Hecke algebra attached to $\Phi$) is part of \cite{CK1}.

\medskip

\noindent
{\bf Algorithm} (\cite{Sl}). Start with an empty bipartition
$(\mu,\mu')$, $\mu=\emptyset$, $\mu'=\emptyset$. Locate the largest
content in absolute value in $\lambda+\ep$. This could appear in the
last box of the first row or the last box of the first column. Assume
first that these two entries are distinct. If the largest content is
in the first row, remove the row from $\lambda$ and put its length in
$\mu$. If the largest content is in the first column, remove the
column from $\lambda$ and put its length in $\mu'$. Continue with the
remaining (smaller) partition. 

If at this step there was an ambiguity, namely the largest entry in
absolute value appeared twice, start two cases, and proceed
in each one of them separately as above.  

If at the last step, we are left with a single box, treat it as a row
if its content is nonnegative, and as a column if its content is
negative.

It is apparent that the number of $W$-types that the algorithm returns
is always a power of $2$. Moreover, it is known that the algorithm
returns a unique $W_n$-type if and only if the associated nilpotent
orbit is distinguished.
 
\medskip

Let $\sigma_{(\mu,\mu')}$ be one of the $W_n$-types returned by the
algorithm. It remains to check that the multiplicity of $\wti\sigma_\lambda$ in $\sigma_{(\mu,\mu')}\otimes S$ is nonzero (claim (2) of
Theorem \ref{t:intro}). We have:
\begin{equation}
\Hom_{\wti W}[\wti\sigma_\lambda,\sigma_{(\mu,\mu')}\otimes
S]=\Hom_{\wti W}[\sigma_{(\lambda,\emptyset)}\otimes
S,\sigma_{(\mu,\mu')}\otimes S]=\Hom_W[\bigwedge^\bullet
V,\sigma_{(\lambda,\emptyset)}\otimes \sigma_{(\mu,\mu')}].
\end{equation}
In bipartition notation, we have $\bigwedge^k
V=\sigma_{((n-k),(1^k))}$. We claim that $\bigwedge^k
V$, where $k=|\mu|$, occurs in $\sigma_{(\lambda,\emptyset)}\otimes
\sigma_{(\mu,\mu')}.$ To see this, notice that (\ref{indBn}) implies:
\begin{align*}
\sigma_{(\lambda,\emptyset)}\otimes\sigma_{((n-k),(1^k))}&=\sigma_{(\lambda,\emptyset)}\otimes\Ind_{S_{n-k}\times
S_k\times (\bZ/2\bZ)^n}^{W_n}(\triv\boxtimes\sgn\boxtimes
(\triv^{n-k}\boxtimes\sgn^k))\\
&=\Ind_{S_{n-k}\times S_k\times
  (\bZ/2\bZ)^n}^{W_n}(\sigma_\lambda|_{S_{n-k}\times S_k}\otimes
(\triv\boxtimes\sgn)\boxtimes (\triv^{n-k}\boxtimes\sgn^k)).
\end{align*}
From the construction of $\mu$ (using the rows of $\lambda$) and
$\mu'$ (using the columns of $\lambda$), and the induction/restriction
rules in $S_n$, we have:
\begin{equation*}
\sigma_\mu\boxtimes\sigma_{{\mu'}^t}\hookrightarrow
\sigma_\lambda|_{S_{n-k}\times S_k},
\end{equation*}
where ${\mu'}^t$ denotes the transpose partition of $\mu'$. Since
$\sigma_{{\mu'}^t}=\sigma_{\mu'}\otimes\sgn$, as
$S_k$-representations, the claim follows by applying (\ref{indBn}) again.

\subsection{Type $D$} Let $W(D_n)$ denote the Weyl group of type $D_n$. This is a subgroup of $W_n$ of index 2. An irreducible $W_n$-representation $\sigma_{(\lambda,\mu)}$ restricts to a representation of $W(D_n)$ which is:

\begin{itemize}
\item irreducible if $\lambda\neq\mu$;
\item a sum of two inequivalent irreducible $W(D_n)$-types if $\lambda=\mu$. 
\end{itemize}

Moreover, $\sigma_{(\lambda,\mu)}$ and $\sigma_{(\mu,\lambda)}$, $\lambda\neq\mu$, restrict to the same $W(D_n)$-representation. In our case, the only $W(D_n)$-types that we consider are those associated via Springer's correspondence with nilpotent orbits in $\C N_0(D_n)$. In the Jordan form partition notation (\cite{Ca2}), these are the orbits indexed by partitions of $2n$ where all parts are odd and appear with multiplicity at most $2$. It turns out that for every Springer representation attached to one of these orbits, the bipartition $(\lambda,\mu)$ has the property that  $\lambda\neq\mu.$

\smallskip

The classification of $\widehat{\wti W(D_n)}_{\mathsf{gen}}$ is obtained from Theorem \ref{t:Bn}. We define the equivalence relation $\sim$ on $P(n)$: $\lambda\sim \lambda^t$. This is the combinatorial equivalent of the relation $\sim$ on $\widehat S_n$: $\sigma_\lambda\sim \sigma_{\lambda}\otimes \sgn=\sigma_{\lambda^t}.$

\begin{theorem}[{\cite[\S7,8]{Re}}]\label{t:Dn} There is a one-to-one correspondence 
$$\widehat{\wti W(D_n)}_{\mathsf{gen}}/_\sim\longleftrightarrow P(n)/_\sim.$$
More precisely, recall the irreducible representations $\sigma_\lambda$ of $S_n$ and $\wti\sigma_\lambda$ of $\wti W_n$. We have:
\begin{enumerate}
\item If $n$ is odd, every irreducible genuine $\wti W_n$-representations $\wti\sigma_\lambda$ restricts to an irreducible $\wti W(D_n)$-representation, and this gives a complete set of inequivalent irreducible genuine $\wti W(D_n)$-representations. Moreover, $\wti\sigma_\lambda$ and $\wti\sigma_{\lambda^t}$ are associate as $\wti W(D_n)$-representations.
\item If $n$ is even, and if $\lambda=\lambda^t$, then $\wti\sigma_\lambda$ restricts to a sum of two associate irreducible $\wti W(D_n)$-type; if $\lambda\neq \lambda^t,$ then $\wti\sigma_\lambda$ restricts to an irreducible $\wti W(D_n)$-representation.
\end{enumerate}
\end{theorem}

The analysis of the map $\Psi$ is now completely analogous to the cases $B_n/C_n$. The same formulas and algorithm hold with the convention that $\epsilon=0.$

\subsection{Exceptional root systems}\label{sec:exc} The character tables for
exceptional $\wti W$ are in \cite{Mo2}. We reorganize them here so
that the claims in Theorem \ref{t:intro} follow. The notation for
$\wti W$-types is borrowed from \cite{Mo2}.  We put a $*$ next to $\wti
W$-type to indicated that it has an associate $\wti W$-type. We use
Carter's notation for $W$-types and nilpotent orbits (\cite{Ca2}). In each table, we give the correspondences (1), (2) in Theorem \ref{t:intro} between genuine $\wti W$-types, nilpotent orbits, and Springer representations attached to nilpotent orbits. We also give the traces of the characters of the genuine $\wti W$-types on elements of the form $\wti s_\al\wti s_\beta$, where $\al,\beta$ form an $A_2,$ $B_2$, or $G_2$. Using these traces and the sizes of the corresponding conjugacy classes (which are listed in \cite{Ca}), we computed $\wti\sigma(\Omega_{\wti W})$ and verified assertion (1) in Theorem \ref{t:intro}. For tensor product decompositions, we used the package {\it chevie} in the computer algebra system GAP.

\begin{table}[h]
\caption{$G_2$\label{t:G2}}
\begin{tabular}{|c|c|c|c|c|}
\hline
Nilpotent $e\in \C N_0$ &$\sigma_{e,\phi}\in \widehat W$
&$\wti\sigma\in \widehat{\wti W}_{\mathsf{gen}}$ &$\tr_{\wti \sigma}(A_2)$
&$\tr_{\wti \sigma}(G_2)$\\
\hline
$G_2$ &$(1,0)$ &$2_s$ &$1$ &$\sqrt 3$\\
\hline
$G_2(a_1)$ &$(2,1)$ &$2_{sss}$ &$-2$ &$0$\\
          &$(1,3)'$ &$2_{ss}$ &$1$ &$-\sqrt 3$\\
\hline
\end{tabular}
\end{table}

\begin{table}[h]
\caption{$F_4$\label{t:F4}}
\begin{tabular}{|c|c|c|c|c|c|}
\hline
Nilpotent $e\in \C N_0$ &$\sigma_{e,\phi}\in \widehat W$
&$\wti\sigma\in\widehat{\wti W}_{\mathsf{gen}} $ &$\tr_{\wti \sigma}(A^l_2)$ &$\tr_{\wti \sigma}(A^s_2)$
&$\tr_{\wti \sigma}(B_2)$\\
\hline
$F_4$ &$(1,0)$ &$4_s$ &$2$ &$2$ &$2\sqrt 2$\\
\hline
$F_4(a_1)$ &$(4,1)$ &$12_s$ &$0$ &$0$ &$2\sqrt 2$\\
          &$(2,4)''$  &$8_{sss}$ &$4$ &$-2$ &$0$\\  
\hline
$F_4(a_2)$ &$(9,2)$ &$24_s$ &$0$ &$0$ &$0$\\
          &$(2,4)'$  &$8_{ssss}$ &$-2$ &$4$ &$0$\\
\hline
$F_4(a_3)$&$(12,4)$ &$8_{ss}$ &$0$ &$0$ &$-2\sqrt 2$\\
          &$(9,6)'$ &$12_{ss}$ &$-2$ &$-2$ &$0$\\ 
          &$(6,6)''$ &$8_{s}$ &$-2$ &$-2$ &$0$\\
          &$(1,12)'$ &$4_{ss}$ &$2$ &$2$ &$-2\sqrt 2$\\
\hline
\end{tabular}
\end{table}

\begin{table}[h]
\caption{$E_6$\label{t:E6}}
\begin{tabular}{|c|c|c|c|}
\hline
Nilpotent $e\in \C N_0$ &$\sigma_{e,\phi}\in \widehat W$
&$\wti\sigma\in\widehat{\wti W}_{\mathsf{gen}} $ &$\tr_{\wti
  \sigma}(A_2)$\\
\hline
$E_6$ &$(1,0)$ &$8_s$ &$4$\\
\hline
$E_6(a_1)$ &$(6,1)$ &$40_s$ &$8$\\
\hline
$E_6(a_3)$ &$(30,3)$ &$120_s$ &$0$\\
           &$(15,5)$ &$72_s$ &$0$\\
\hline
$D_5$ &$(20,2)$ &$60_s*$ &$6$\\
\hline
$D_5(a_1)$ &$(64,4)$ &$80_s*$ &$-2$\\
\hline
$A_4+A_1$ &$(60,5)$ &$64_s*$ &$-4$\\
\hline
$D_4(a_1)$ &$(80,7),(90,8),(20,10)$    &$40_{ss}$ &$-4$\\
&$(90,8)$ &$20_s*$ &$-2$ \\
           
\hline
\end{tabular}
\end{table}

\begin{table}[h]
\caption{$E_7$\label{t:E7}}
\begin{tabular}{|c|c|c|c|}
\hline
Nilpotent $e\in \C N_0$ &$\sigma_{e,\phi}\in
\widehat W$
&$\wti\sigma\in \widehat{\wti W}_{\mathsf{gen}}$ &$\tr_{\wti
  \sigma}(A_2)$\\
\hline
$E_7$ &$(1,0)$ &$8_s*$ &$4$\\
\hline
$E_7(a_1)$ &$(7,1)$ &$48_s*$ &$12$\\
\hline
$E_7(a_2)$ &$(27,2)$ &$168_s*$ &$24$\\
\hline
$E_7(a_3)$ &$(56,3)$ &$280_s*$ &$20$\\
           &$(21,6)$ &$112_s*$ &$8$\\
\hline
$E_7(a_4)$ &$(189,5)$ &$720_s*$ &$0$ \\
           &$(15,7)$ &$120_s*$ &$0$\\
\hline
$E_7(a_5)$ &$(315,7)$ &$448_{s}*$ &$-4$\\
          &$(280,9)$ &$560_s*$ &$-20$\\
          &$(35,13)$ &$112_{ss}*$ &$-16$\\
\hline
$E_6(a_1)$ &$(120,4),(105,5)$ &$512_s*$ &$16$\\
\hline
$A_4+A_1$ &$(512,11),(512,12)$ &$64_s*$ &$-4$\\
          &$(512,11),(512,12)$ &$64_{ss}*$ &$-4$\\
\hline
\end{tabular}
\end{table}

\begin{table}[h]
\caption{$E_8$\label{t:E8}}
\begin{tabular}{|c|c|c|c|}
\hline
Nilpotent $e\in \C N_0$ &$\sigma_{e,\phi}\in
\widehat W$
&$\wti\sigma\in \widehat{\wti W}_{\mathsf{gen}}$ &$\tr_{\wti
  \sigma}(A_2)$\\
\hline
$E_8$ &$(1,0)$ &$16_s$ &$8$\\
\hline

$E_8(a_1)$ &$(8,1)$  &$112_s$ &$32$\\
\hline

$E_8(a_2)$ &$(35,2)$ &$448_{ss}$ &$80$\\
\hline

$E_8(a_3)$ &$(112,3)$ &$1344_{ss}$ &$168$\\
&$(28,8)$ &$320_s$ &$40$\\
\hline

$E_8(a_4)$ &$(210,4)$ &$2016_s$ &$144$\\
&$(160,7)$ &$1680_s$ &$120$\\
\hline

$E_8(a_5)$ &$(700,6)$ &$5600_{sss}$ &$160$\\
           &$(300,8)$ &$2800_s$ &$80$\\
\hline

$E_8(a_6)$ &$(1400,8)$ &$6480_s$ &$0$\\
&$(1575,10)$ &$9072_s$ &$0$\\
&$(350,14)$ &$2592_s$ &$0$\\
\hline

$E_8(a_7)$ &$(4480,16)$
&$896_{s}$ &$-72$\\
&$(5670,18)$ &$2016_{sss}$ &$-72$\\
&$(4536,18)$ &$2016_{ss}$ &$-48$\\
&$(1680,22)$ &$1344_s$ &$-40$\\
&$(1400,20)$ &$1120_s$ &$-32$\\
&$(70,32)$ &$224_s$ &$-8$\\
\hline

$E_8(b_4)$ &$(560,5)$ &$5600_{ss}$ &$280$\\
 &$(50,8)$ &$800_s$ &$40$\\
\hline

$E_8(b_5)$ &$(1400,7)$ &$6720_s$ &$128$\\ 
&$(1008,9)$ &$7168_s$ &$120$\\
&$(56,19)$ &$448_s$ &$8$ \\
\hline

$E_8(b_6)$ &$(2240,10)$ &$8400_s$ &$-120$\\
&$(840,13)$ &$5600_s$ &$-80$\\
&$(175,12)$ &$2800_{ss}$ &$-40$\\
\hline

$D_5+A_2$ &$(4536,13),(840,13)$ &$4800_s$ &$-120$\\
\hline

$D_7(a_1)$ &$(3240,9),(1050,10)$ &$11200_s$ &$-40$\\

\hline

$D_7(a_2)$ &$(4200,12),(3360,13)$ &$7168_{ss}$ &$-160$\\

\hline

$E_6(a_1)+A_1$ &$(4096,11),(4096,12)$ &$8192_s$ &$-128$\\

\hline
\end{tabular}
\end{table}

\begin{remark}\label{counter} In
$E_6$, there is a nilpotent orbit $D_4(a_1)\subset \C N_0(E_6)$, such
that $A(e)=S_3.$ There are three Springer representations attached to
$D_4(a_1)$: $\sigma_{e,(3)}$, $\sigma_{e,(21)}$, and
$\sigma_{e,(111)}.$ Since the rank is even, there is a single spin
module $S.$ The fiber $\Psi^{-1}(e)$ contains three $\wti W$-types
$\wti\sigma_1$, $\wti \sigma_2,$ $\wti\sigma_3$ of dimensions $40$,
$20$, $20$, respectively. Moreover, $\wti\sigma_2$ and $\wti\sigma_3$
are associate.  We compute that in the decomposition
$\wti\sigma_1\otimes S$ occur all three $\sigma_{e,(3)}$, $\sigma_{e,(21)}$, and
$\sigma_{e,(111)},$ while in the decomposition of $\wti\sigma_2\otimes
S$ (equivalently $\wti\sigma_3\otimes S$) only $\sigma_{e,(21)}$
occurs among the three Springer representations. 
\end{remark}

\subsection{The generalized Springer correspondence}\label{s:gen}
The references for the construction of the generalized Springer
correspondence, and for the results we need to use 
are \cite{L,L3}.

Let $G$ be a simply connected
complex simple group, with a fixed Borel subgroup $B$,
and maximal torus $H\subset B$. The Lie algebras will be denoted by
the corresponding Gothic letter. Fix a nondegenerate $G$-invariant symmetric bilinear form
$\langle~,~\rangle$ on $\fg.$ Denote also by $\langle~,~\rangle$ the
dual form on $\fg^*.$

\begin{definition}
A cuspidal triple for $G$ is a triple $(L,\C C,\C L),$ where $L$ is a
Levi subgroup of $G,$ $\C C$ is a nilpotent $L-$orbit on the Lie
algebra $\fk l$, and $\C L$ is an irreducible $G-$equivariant local
system on $\C C$ which is cuspidal in the sense of \cite{L,L3}. 
\end{definition}

Let us fix a cuspidal triple $(L,\C C,\C L)$, such that $H\subset L,$
and $P=LU\supset B$ is a parabolic subgroup. Let
$T\subset L$ denote the identity component of the center of $L$, with Lie
algebra $\fk t$. Write an orthogonal decomposition with respect to $\langle~,~\rangle$,
$\fh=\fk t +\fk a; $
here $\fk a$ is a Cartan subalgebra for the semisimple part of $\fk
l.$ Let 
\begin{equation}\label{projt}
\pr_{\C C}:\fh \to \fk t
\end{equation}
denote the corresponding projection onto $\fk t.$

Following \cite[\S2]{L3}, we attach to $(L,\C C,\C L)$ an $\bR$-root system
  $\Phi=(V,R,V^\vee,R^\vee)$ and a $W$-invariant function $c:R^+\to\bZ$. Let $R\subset \fk t^*$ be the reduced part of the root system given by
  the nonzero weights of  $\text{ad}(\ft)$ on $\fg$; it can be
  identified with the root system of the reductive part of $Z_G(x),$
  where $x\in \C C$.  Let $V$ be the $\bR$-span of $R$ in $\fk t^*.$  
The Weyl group is 
\begin{equation}
W=N_G(L)/L.
\end{equation}
 This is a Coxeter group due to the particular form $L$
must have to allow a cuspidal local system. 

Let $R^+$ is the subset of $R$ for which the
  corresponding weight space lives in $\fk u$. The simple roots $\Pi=\{\al_i:i\in I\}$ correspond to the Levi
  subgroups $L_i$ containing $L$ maximally: $\al_i$ is the unique
  element in $R^+$ which is trivial on the center of $\fk l_i$.  
For every simple $\al_i$, $c({\al_i})\ge 2$ is defined to be
  the smallest integer  such that
\begin{equation}
\text{ad}(x)^{c({\al_i})-1}:\fk l_i\cap\fk u\to \fk l_i\cap\fk u
\text{ is zero.} 
\end{equation} 
This is a $W$-invariant function on $\Pi$ and we extend it to $R^+.$

The  explicit classification of cuspidal triples (when $G$ is simple),
  along with the corresponding values for the parameters $c(\al)$ can
  be found in the tables of \cite[\S2.13]{L3}.

Define $R^\perp=\{x\in\fk t:~\al(x)=0,\text{ for all }\al\in R\}$, and
let $\fk t'$ be the orthogonal complement of $R^\perp$ in $\fk t$. For every
$\al\in R,$ define $\check\al\in \fk t$ to be the unique element of
$\fk t'$ such that $\al(\check\al)=2.$ Let $R^\vee$ denote the set of
all $\check\al,$ and let $V^\vee$ be the $\bR$-span of $R^\vee$ in
$\fk t'.$

\smallskip

Consider  the variety
\begin{align}
&\dot{\fg}=\{(x,gP)\in \fg\times G/P: ~\text{Ad}(g^{-1})x\in \C
  C+\fk t+\fk u\},
\end{align} 
on which $G\times \bC^\times$ acts via $(g_1,\lambda)$: $x\mapsto
\lambda^{-2}\text{Ad}(g_1)x,$ $x\in \fg,$ and $gP\mapsto g_1gP,$ $g\in
G.$ Let $\pr_1$ and $\pr_2$ be the projections of $\dot{\fg}$ on the
two factors. For every nilpotent element $e\in\fg$, let $\C P_e$
denote the preimage of $\{e\}$ in $\dot{\fg}$ under $\pr_1.$ Via
$\pr_2$, we may make the identification:
$$\CP_e=\{gP:~ Ad(g^{-1})e\in \C C+\fk u\}.$$

Let $\dot{\C L}$ denote the pull-back of the local system $\C L$ on
$\C C$ under the $G\times \bC^\times$-equivariant projection $\dot\fg\to \C C$, $(x,gP)\mapsto \text{Ad}(g^{-1})x$.

Let $\pi_1(e)=Z_G(e)/Z_G(e)^0$ be the fundamental group of $G\cdot e.$ 
The hypercohomology with compact support
$H^\bullet_c(\C P_e,\dot{\C L})$ carries a 
$\pi_1(e)\times W$ action (\cite{L}, see also \cite{L3}). Let $\widehat {A(e)}_\C C$
denote the set of irreducible representations of $\pi_1(e)$ which appear in
this way. Moreover, for $\phi\in\widehat{A(e)}_{\C C}$ we have:
\begin{equation}\label{eq:3.5.3}
\sigma_{(e,\phi)}:=\operatorname{Hom}_{\pi_1(e)}[\phi, H^{2\dim\C P_e}_c(\C P_e,\dot{\C L})]
\end{equation}
is an irreducible $W-$representation. The
correspondence $\sqcup_{e\in G\backslash \C N(\fg)}\widehat {A(e)}_\C C\to \widehat W,$ 
$(e,\phi)\to \sigma_{(e,\phi)}$ is the generalized Springer correspondence
of \cite{L}, and it is a bijection. We normalize it so that
$\sigma_{(e,\phi)}=\sgn,$ when $e\in G\cdot \C C$ (there is single
$\phi$ that appears in that case).

\begin{definition}\label{d:solvC}
Denote $\C T_0(G,\C C)=\{[(e,h,f)]\in\C T_0:~ \widehat {A(e)}_{\C C}\neq 0\}.$
\end{definition}

\begin{theorem}\label{t:gen}
There is a surjective map
\begin{equation}\label{e:classC}
\Psi_c:\widehat{\wti W}_{\mathsf{gen}} \longrightarrow \C T_0(G,\C C),
\end{equation}
($\C T_0(G,\C C)$ is as in Definition \ref{d:solvC}) with the following properties:
\begin{enumerate}
\item If $\Psi(\wti\sigma)=[(e,h,f)]$, then 
\begin{equation}
\wti\sigma(\Omega_{\wti W,c})=\langle \pr_{\C
  C}(h),\pr_{\C C}(h)\rangle.
\end{equation}
where $\Omega_{\wti W,c}$ is as in (\ref{omWtilde}) and $\pr_{\C C}$
is as in (\ref{projt});
\item Let $(e,h,f)\in \C T_0(G,\CC)$ be given. For every  Springer representation $\sigma_{(e,\phi)}$,
  $\phi\in\widehat {A(e)}_{\C C}$, and every spin $\wti W$-module $S$,
  there exists $\wti \sigma\in \Psi^{-1}[(e,h,f)]$ such that $\wti\sigma$ appears with
  nonzero multiplicity in the tensor product
  $\sigma_{(e,\phi)}\otimes S$. Conversely, for every $\wti\sigma\in \Psi^{-1}[(e,h,f)]$, there exists a
  spin $\wti W$-module $S$ and a Springer representation
  $\sigma_{(e,\phi)}$, such that $\wti\sigma$ is contained in
  $\sigma_{(e,\phi)}\otimes S.$
\item If $e$ is distinguished, then properties (1) and (2) above
  determine a   bijection:
\begin{equation}\label{distC}
\Psi^{-1}([e,h,f])/_\sim \longleftrightarrow \{\sigma_{e,\phi}:
\phi\in \widehat {A(e)}_{\C C}\}.
\end{equation}
\end{enumerate}
\end{theorem}

\begin{proof}
Again the proof amounts to verifying the assertions in every case. The
interesting cases are when $R$ is of type $B_n/C_n$, $G_2$, or
$F_4.$

For type $G_2$, the cuspidal local system appears for the Levi
subgroup $L=2A_2$ in $G=E_6$. Assuming that the simple roots of $R$ of type
$G_2$ are $\al,\beta$ with $\al$ long and $\beta$ short, the
parameter function is $c(\al)=1$, $c(\beta)=3$. To give an explicit
formula for $\Omega_{\wti W,c}$, let's normalize $\al,\beta$ such that
such that $\langle\check\al,\check\al\rangle=2$ and
$\langle\check\beta,\check\beta\rangle=6$. Then we have
\begin{equation}
\frac 14\wti\sigma(\Omega_{\wti W,c})=\frac 32 (c(\al)^2+3 c(\beta)^2)+4\sqrt 3 c(\al) c(\beta)\frac{\tr_{\wti\sigma}(G_2)}{\dim\wti\sigma}+(c(\al)^2+3c(\beta)^2)\frac{\tr_{\wti\sigma}(A_2)}{\dim\wti\sigma};
\end{equation}
recall that the notation for conjugacy classes  is as in
\cite{Ca}. Then one can verify easily part (1) of Theorem
\ref{t:gen}. 

For type $F_4$, the cuspidal local system appears for the Levi
subgroup $L=(3A_1)'$ in $G=E_7$. Assuming that the simple roots of $R$
of type $F_4$ are $\al,\beta$ with $\al$ short and $\beta$ long, the
parameter function is $c(\al)=2$, $c(\beta)=1$. We normalize
$\al,\beta$ so that $\langle\check\al,\check\al\rangle=2$ and
$\langle\check\beta,\check\beta\rangle=1.$ Then we have:
\begin{equation}
\frac 14\wti\sigma(\Omega_{\wti W,c})=3(2c(\al)^2+c(\beta)^2)+16
c(\al)^2\frac{\tr_{\wti\sigma}(A_2)}{\dim\wti\sigma}+8
c(\beta)^2\frac{\tr_{\wti\sigma}(\wti A_2)}{\dim\wti\sigma}+18\sqrt 2c(\al)c(\beta)\frac{\tr_{\wti\sigma}(B_2)}{\dim\wti\sigma}.
\end{equation}

The correspondence from Theorem \ref{t:gen} for $G_2$ and $F_4$ is listed in Tables
\ref{t:G2C} and \ref{t:F4C}. The explicit values of $\pr_{\C C}(h)$ in coordinates, for these cases, can be found in \cite[Tables 1 and 2]{Ci}.

\begin{table}[h]
\caption{$2A_2$ in $E_6$, $W=G_2$.\label{t:G2C}}
\begin{tabular}{|c|c|c|}
\hline
Nilpotent $e\in \C N_{0,C}$ &$\sigma_{e,\phi}\in \widehat W$
&$\wti\sigma\in \widehat{\wti W}_{\mathsf{gen}}$ \\
\hline
$E_6$ &$(1,0)$ &$2_s$ \\

\hline
$E_6(a_1)$          &$(1,3)'$ &$2_{ss}$ \\
\hline
$E_6(a_3)$ &$(2,1)$ &$2_{sss}$ \\
\hline
\end{tabular}
\end{table}

\begin{table}[h]
\caption{$(3A_1)'$ in $E_7$,  $W=F_4$.\label{t:F4C}}
\begin{tabular}{|c|c|c|}
\hline
Nilpotent $e\in \C N_{0,C}$ &$\sigma_{e,\phi}\in \widehat W$
&$\wti\sigma\in\widehat{\wti W}_{\mathsf{gen}} $ \\
\hline
$E_7$ &$(1,0)$ &$4_s$\\
\hline 
$E_7(a_1)$         &$(2,4)''$  &$8_{sss}$ \\  
\hline
$E_7(a_2)$ &$(4,1)$ &$12_s$ \\
\hline
 
$E_7(a_3)$ &$(8,3)'$ &$24_s$ \\  
&$(1,12)'$ &$4_{ss}$\\
\hline
$E_7(a_4)$          &$(2,4)'$  &$8_{ssss}$ \\
         &$(4,7)'$ &$12_{ss}$ \\ 
\hline
$E_7(a_5)$&$(12,4)$ &$8_{ss}$ \\     
&$(6,6)''$ &$8_{s}$ \\
\hline
 
\end{tabular}
\end{table}

For types $B_n$ and $C_n$, the combinatorics is similar to that in the
proof of Theorem \ref{t:intro}. More precisely, assume the notation from Theorem \ref{t:Bn} and the discussion following it. Assume also
that the roots of type $B_n$ and $C_n$ are in the standard Bourbaki coordinates. For a partition $\lambda$ of $n$, viewed as a left justified Young tableau, define the content of the box $(i,j)$ with parameters $c_1,c_2$ to be the number $m_c(i,j):=c_1(i-j)+c_2.$ Let us denote by $p_2(\lambda,c_1,c_2)$ the sum of squares of contents of boxes for $\lambda$ and parameters $c_1,c_2$.  The same computation as after Theorem \ref{t:Bn} shows that
\begin{equation}
\begin{aligned}
&\wti\sigma_\lambda(\Omega_{\wti W,c})=4p_2(\lambda,c(\ep_1-\ep_2),c(\ep_n)),\text{ for type $B_n$},\\
& \wti\sigma_\lambda(\Omega_{\wti W,c})=4p_2(\lambda,c(\ep_1-\ep_2),\frac 12c(2\ep_n)),\text{ for type $C_n$}.
\end{aligned}
\end{equation}
Notice that $\wti\sigma_\lambda(\Omega_{\wti W,c})$ for $C_n$ is identical with $\wti\sigma_\lambda(\Omega_{\wti W,c})$ for $B_n$ if we set $c(2\ep_n)=2 c(\ep_n).$ The geometric values of the parameters are as follows (\S2.13 in \cite{L3}):

\begin{enumerate}
\item $\fg=sp(2k+2n)$, $\fk l=sp(2k)\oplus\bC^n$, $\C C=(2,4,\dots,2p)\oplus 0$, $k=p(p+1)/2$, $\Phi=B_n$, $c(\ep_1-\ep_2)=2$, $c(\ep_n)=2p+1$;
\item $\fg=so(k+2n)$, $\fk l=so(k)\oplus\bC^n$, $\C C=(1,3,\dots,2p+1)\oplus 0$, $k=p^2$, $\Phi=B_n$, $c(\ep_1-\ep_2)=2$, $c(\ep_n)=2p+2$;
\item $\fg=so(k+4n)$, $\fk l=so(k)\oplus sl(2)^n\oplus \bC^n$, $\C C=(1,5,9,\dots, 4p+1)\oplus (2)^n\oplus 0$, $k=(p+1)(2p+1)$, $c(\ep_1-\ep_2)=4$, $c(\ep_n)=4p+3$;
\item $\fg=so(k+4n)$, $\fk l=so(k)\oplus sl(2)^n\oplus \bC^n$, $\C C=(3,7,11,\dots, 4p+3)\oplus (2)^n\oplus 0$, $k=(p+1)(2p+3)$, $c(\ep_1-\ep_2)=4$, $c(\ep_n)=4p+5$.
\end{enumerate}

For the partition $\lambda$ with content $m_c(i,j)$, Slooten's algorithm is the same as in the case $c\equiv 1$. Also the analysis of the tensor product decomposition works in the same way as in the proof of Theorem \ref{t:intro} for $B_n$ and $c\equiv 1$. (The argument there only uses Slooten's algorithm and the Weyl group of type $B_n$, and not the parameter function $c$.)  The algorithm giving the nilpotent element $e\in \fg$ from the partition $\lambda$ with content $m_c(i,j)$, at geometric values of $c$, is again very similar to the one in \cite[\S4.4]{CK2}, and we skip the details.

\end{proof}

\ifx\undefined\bysame
\newcommand{\bysame}{\leavevmode\hbox to3em{\hrulefill}\,}
\fi

\end{document}